\documentclass[12pt,a4paper]{amsart}
\usepackage{amsfonts,amssymb,latexsym,amsmath, amsxtra, eufrak}
\usepackage[all]{xy}
\usepackage{graphics,graphicx,epsfig}

\usepackage[OT2,T1]{fontenc}
\DeclareSymbolFont{cyrletters}{OT2}{wncyr}{m}{n}
\DeclareMathSymbol{\Sha}{\mathalpha}{cyrletters}{"58}

\theoremstyle{plain}
\newtheorem{theorem}{Theorem}[section]

\newtheorem{proposition}[theorem]{Proposition}

\newtheorem*{conjecture*}{Conjecture}

\newtheorem{definition}[theorem]{Definition}

\numberwithin{equation}{section}

\let\non\nonumber

\newcommand{\bea}{\begin{eqnarray}}
\newcommand{\eea}{\end{eqnarray}}
\newcommand{\be}{\begin{equation}}
\newcommand{\ee}{\end{equation}}

\newcommand{\sgn}{\mathrm{sgn}}

\newcommand{\noi}{\noindent}

\newcommand{\im}{\mathrm{Im}}

\newcommand{\bP}{\mathbb{P}}

\newcommand{\CI}{\mathcal{I}} 
\newcommand{\CM}{\mathcal{M}}
\newcommand{\CN}{\mathcal{N}}
\newcommand{\CO}{\mathcal{O}}
\newcommand{\CS}{\mathcal{S}}
\newcommand{\CZ}{\mathcal{Z}}

\newcommand{\bs}{\boldsymbol}

\begin{document}

\title[Sheaves on $\bP^2$ and generalized Appell functions]{Sheaves on $\bP^2$ and generalized Appell functions}

\author{Jan Manschot}
\address{Institut Camille Jordan\\  
Universit\'e Claude Bernard Lyon 1\\
43 boulevard du 11 novembre 1918\\
69622 Villeurbanne cedex\\ France \vspace{.2cm} \newline 
Current address: School of Mathematics, Trinity College, College Green, Dublin 2, Ireland}

\vskip 2 cm
 
\begin{abstract}
\baselineskip=18pt
\noi 
A closed expression is given for the generating function of (virtual) Poincar\'e polynomials of moduli spaces of semi-stable sheaves on the projective plane $\bP^2$ with arbitrary rank $r$ and Chern classes. This generating function is known to equal the partition function of topologically twisted gauge theory with $\CN=4$ supersymmetry and gauge group $U(r)$, which localizes  on the Hermitian Yang-Mills solutions of the gauge field. To classify and study the novel generating functions, the notion of Appell functions with signature $(n_+,n_-)$ is introduced. For $n_-=1$, these novel functions reduce to the known class of Appell functions with multiple variables or higher level. 
\end{abstract}
\maketitle

\baselineskip=19pt

\tableofcontents

\vspace{0.5cm}

\section{Introduction}
Moduli spaces and their topological invariants are of fundamental interest for mathematics and physics. The Donaldson-Uhlenbeck-Yau theorem \cite{Donaldson:1987, Uhlenbeck:1986} rigorously establishes the correspondence between moduli spaces of instanton solutions in Yang-Mills theory and moduli spaces of semi-stable vector bundles and sheaves. Much progress is made in recent years on the properties and computation of topological invariants of these moduli spaces for rational and ruled algebraic surfaces. Among the used techniques for this progress are wall-crossing \cite{Klemm:2012sx,Manschot:2010nc,
  Manschot:2011ym, Yoshioka:1994, Yoshioka:1995, Yoshioka:1996}, toric localization in moduli spaces \cite{Kool:2009, weist:2009}, and the Hall algebra \cite{Mozgovoy:2013zqx}.  
 
Generating functions of topological
invariants of moduli spaces exhibit often interesting arithmetic and
modular properties. On the arithmetic side, the topological invariants are known to equal 
counts of (colored) partitions for rank 1 sheaves \cite{Gottsche:1990}, dimensions of representations of the Mathieu group for the K3 surface \cite{Cheng:2012uy, Katz:2014}, and class numbers for rank 2 sheaves on the projective plane $\bP^2$ \cite{Klyachko:1991, Vafa:1994tf, Zagier:1975}. 
The appearance of modularity is understood physically by the relation to gauge theory, and the $SL(2,\mathbb{Z})$ electric-magnetic duality group of this theory. The path integral of topologically twisted $\CN=4$ supersymmetric Yang-Mills theory with gauge group $U(r)$ (also known as Vafa-Witten partition function) can
be shown to enumerate basic topological invariants as the Euler number and Poincar\'e polynomial of the moduli spaces of
vector bundles of rank $r$ \cite{Vafa:1994tf}. The electric-magnetic duality group then implies 
modular  transformation properties for the generating functions of these invariants.
 
The class of rational and ruled surfaces allows to study the generating functions very explicitly, and in particular their dependence on the polarization or K\"ahler 2-form $J$. In the limit of vanishing volume of the base of the ruled surface, the generating functions of the topological invariants take the form of a beautiful infinite product formula \cite{Manschot:2011ym, Mozgovoy:2013zqx} which transforms as a Jacobi form \cite{Eichler:1985} under modular transformations. Application of wall-crossing formulas \cite{Joyce:2004, Joyce:2008, Kontsevich:2008} allows to determine the invariants for other choices of the polarization
$J$. The change of polarization is taken into account by so-called indefinite theta functions \cite{Gottsche:1996, Manschot:2010nc, Zwegers:2000}. These are convergent and holomorphic sums over a subset of an indefinite lattice of signature $(r-1,r-1)$. They typically destroy the nice modular transformation properties. However using the theory of mock modular forms \cite{Zwegers:2010}, one can add a specific non-holomorphic completion such that the modular properties are restored for rank 2 \cite{ Manschot:2011dj, Vafa:1994tf}.  The non-holomorphic completion is however not very well understood from gauge theory, and the holomorphic anomaly equation is only conjectured for $r>2$ \cite{Minahan:1998vr}.  
    
In this brief note, we derive a closed expression for the generating functions for arbitrary rank $r$ and Chern classes. This closed form is given by Equation (\ref{eq:Hrab}) together with Proposition \ref{prop:HrbCaf}. For rank 3, the function simplifies considerably the expression given in \cite{Manschot:2010nc}, and it also allows to relatively quickly determine invariants for $r>3$. The key to this simplification is the fact that the wall-crossing formula of Joyce \cite{Joyce:2004} for virtual Poincar\'e polynomials is very suitable for application in generating functions.\footnote{Toda \cite{Toda:2014} pointed out recently that application of the wall-crossing formula of Joyce for numerical invariants (Euler numbers) \cite{Joyce:2008} is compatable with the theory of indefinite theta functions.} For $r=2$, we find immediately the familiar Appell functions \cite{ Manschot:2011dj, Yoshioka:1994}. To describe the functions for $r\geq 3$, we generalize the notion of the classical Appell function (\ref{eq:appell}). The generalized Appell functions (\ref{eq:genAppell}) are characterized by their signature $(n_+,n_-)$. The functions with signature $(n_+,1)$ reduce to the multi-variable and higher level Appell functions previously described in the literature \cite{Semikhatov:2003uc, Zwegers:2010}.  The novel form of the generating functions is much more suitable for the study of their arithmetic and modular properties than the form in \cite{Manschot:2010nc, Manschot:2011dj}. The properties are currently being determined \cite{work}.  
 
Appell functions have by now a wide variety of applications in number theory, algebraic geometry and mathematical physics \cite{ Cheng:2012uy, Dabholkar:2012nd, Eguchi:1988af, Kac:2001, Manschot:2007ha, Manschot:2009ia, Polishchuk:1998, Semikhatov:2003uc,  Troost:2010ud, Vafa:1994tf, Zwegers:2000}. The functions found here for $\CN=4$ Yang-Mills theory have even more subtle transformation properties, which are also likely to appear in other contexts. We mention only a few here: 
\begin{itemize} 
\item The partition function of Yang-Mills theory on a complex surface is a very useful  model 
  for the more difficult problem of determining partition functions for D4-D2-D0 branes in string theory supported on divisors
  in Calabi-Yau three-folds \cite{Maldacena:1997de}. This problem is relevant for describing quantum black holes in four-dimensional $\CN=2$ supergravity.
 Besides the application to the enumerative geometry of Calabi-Yau three-folds and black hole physics, the modular properties of the generating functions are also important to understand S-duality of
  IIB string theory  \cite{Alexandrov:2012au}. The period integrals which
  appear in the modular completion of the generating function and which render the partition 
  function continuous as function of the stability parameters are expected to be related to twistor integrals occuring in the Darboux coordinates.
  Recently it was proposed that these integrals also occur in the multi-particle Witten index for $\CN=2$ supersymmetric theories in $\mathbb{R}^{3,1}$ \cite{Alexandrov:2014wca}.    
\item Another interesting aspect of the generating functions is that they are expected to appear as partition functions of two dimensional theories. 
 The 6-dimensional M5-brane of M-theory relates $\CN=4$ Yang-Mills to a 2-dimensional field theory on an
  elliptic curve \cite{Gadde:2013sca, Haghighat:2011xx, Maldacena:1997de, Minasian:1999qn, Vafa:1994tf}. It would be interesting if the generalized Appell functions with 
  signature $(n_+,n_-)$ could be derived from this point of view.
\item Application of the wall-crossing for virtual Poincar\'e polynomials (\ref{eq:chCI}) will also simplify the analysis for other complex surfaces and give in this way more examples of the generalized Appell functions. A particularly interesting surface is $\frac{1}{2}$K3 (the rational elliptic surface), for which the Vafa-Witten partition function equals (for suitable $J$) the partition function of topological strings \cite{Minahan:1998vr}. 
\item The usual Appell functions are known to appear as global
  sections of rank 2 bundles on elliptic curves 
  \cite{Polishchuk:1998}. Properties of the Appell functions can be
  understood as $A_\infty$ constraints of the Fukaya category of
  the elliptic curve. It would be interesting to explore whether
  generalized Appell functions have similarly an interpretation as
  global sections of rank $r>2$ bundles on elliptic curves.   
\end{itemize}

The outline of this note is as follows. We start in Section \ref{sec:expaction} by writing the physical action for Hermitean Yang-Mills connections in terms of characteristic classes of the corresponding bundles. Section \ref{sec:change} briefly reviews topological invariants of moduli spaces of semi-stable sheaves and the wall-crossing formula for
virtual Poincar\'e polynomials. Section \ref{sec:genfunctions} defines the generating functions of virtual Poincar\'e polynomials and derives a closed expression for the rational surfaces $\bP^2$ and $\Sigma_1$ and arbitrary rank of the sheaves. In Section 
\ref{sec:appell} we discuss the classical Appell function and generalize it to the larger class of Appell
functions with signature $(n_+,n_-)$. The generating functions of Section \ref{sec:genfunctions} can be expressed in terms of specializations of generalized Appell functions.

\section*{Acknowledgements}
I would like to thank Kathrin Bringmann, Babak Haghighat, Boris Pioline, Artan Shesmani, Cumrun Vafa, Don Zagier,
Miguel Zapata Rol\'on and Sander Zwegers for useful discussions. I am also grateful to the ENS Paris, the Amsterdam String Theory
Workshop 2014 and the Simons Workshop 2014 for their hospitality
during parts of this work.  

\section{The Yang-Mills action in terms of Chern classes}
\label{sec:expaction}
This section evaluates the Yang-Mills action for Hermitean Yang-Mills connections in terms of Chern classes. Section \ref{sec:genfunctions} will use the result to define generating functions for virtual Poincar\'e polynomials of moduli spaces of semi-stable sheaves. The classical Yang-Mills action is given by
\be
\label{action}
\mathcal{S}(A)=-\frac{1}{g^2}\int_S{\rm Tr}\, F\wedge *F +\frac{i\theta}{8\pi^2}\int_S{\rm Tr}\, F\wedge F,
\ee
where $F=dA+A\wedge A=\sum_{a=0}^{N^2-1} F^a t^a$  is the field strength with $t^a$ the Lie algebra generators of $u(N)$ in the adjoint representation. We let $t^0=\frac{1}{N} {\bf 1}_N$, with ${\bf 1}_N$ the $N$-dimensional identity matrix, and $t^a$, $a=1,\dots,N^2-1$, are representing the generators of $SU(N)$. The coupling constant $g$ and the $\theta$-angle are naturally combined in the complexified coupling constant $\tau=\frac{\theta}{2\pi}+\frac{4\pi i}{g^2}$ which takes values in the upper half plane $\mathbb{H}$. 

The partition function of topologically twisted $\mathcal{N}=4$ Yang-Mills theory localizes on the Hermitean Yang-Mills solutions \cite{Vafa:1994tf}. A Hermitean Yang-Mills solution $F$ is such that $F^0$ is a harmonic form and the $F^a$ are anti-self-dual $F^a=-* F^a$. This can also be formulated without reference to a basis of the Lie algebra.  To this end, let $J$ be the K\"ahler form, which is self-dual, $J=*J$, and of degree $(1,1)$. Its norm is positive, $\int_S J\wedge J>0$. Then a Hermitean Yang-Mills connections is a connection such that $\int_S F\wedge J$ is proportional to the identity matrix, and the components of $F$ with form degree $(2,0)$ and $(0,2)$ vanish, $F^{(0,2)}=F^{(2,0)}=0$.

The partition function $\CZ(\tau;J)$ of $\CN=4$ Yang-Mills theory takes then schematically the following form \cite{Vafa:1994tf}
$$
\CZ(\tau;J)=\sum_{k} \chi(\CM_k) \exp(-\CS(A)),
$$
where $k$ is the instanton number and $\chi(\CM(k))$ is the Euler number of the moduli space of instantons with instanton number $k$. Section \ref{sec:change} discusses in more detail the precise definition of $\chi(\CM(k))$. Here we continue with expressing $\exp(-\CS(A))$ in terms of algebraic-geometric data.

The Hitchin-Kobayashi correspondence relates Hermitean Yang-Mills connections to vector bundles. The Chern character $\gamma=(r,c_1,{\rm ch}_2)$ of the vector bundles is given in terms of Yang-Mills data by
$$
r=N, \qquad c_1={\rm ch}_1=\frac{i}{2\pi} {\rm Tr}\,F,\qquad
\frac{1}{2}c_1^2-c_2={\rm ch}_2=-\frac{1}{8\pi^2} {\rm Tr}\,F\wedge F. 
$$
This specificies the proportionality constant between $\int_S F\wedge J$ and ${\bf 1}_N$,
\be
\label{eq:FJ}
\int_S F\wedge J= \frac{2\pi\,{\bf 1}_N}{i\, r} \int_S c_1\wedge J. 
\ee

For the rational and ruled surfaces, the signature of $H^2(S,\mathbb{Z})$ equals $(1,b_2(S)-1)$ or equivalently $b_2^+=1$. Therefore, 2-forms orthogonal to $J$ are anti-self-dual and negative definite. As a result, we can express $*F$ in terms of $F$ and $J$
$$
*F=\frac{2\,F\cdot J\,J}{J^2}-F.
$$ 
This equation together with (\ref{eq:FJ}) gives for the kinetic term
\begin{eqnarray} 
\label{kinetic}
-\int_S {\rm Tr}\, F\wedge *F&=& 8\pi^2 \left( \frac{(c_1\cdot J)^2}{r J^2}  -{\rm ch}_2\right),  
\end{eqnarray}
which is the lower bound of the kinetic energy, given the topological Chern numbers of the Hermitean-Yang-Mills solution. Let $c_{1,\pm}$ be the projections of the $c_1$ to the positive and negative definite subspaces of $H^2(S,\mathbb{Z})$ defined by
\be
\label{eq:plusmin}
c_{1,+}=\frac{c_1\cdot J\,J}{J^2}, \qquad c_{1,-}=c_1-c_{1,+}.
\ee

Using these results, we can express the exponentiated action $\exp(-\mathcal{S}(A))$ in terms of the Chern character and discriminant of the bundle
\begin{eqnarray}
&&\exp\!\left(-2\pi\tau_2\left(-\mathrm{ch_2}+\frac{1}{r}(c_1)_+^2\right)+2\pi i\tau_1 \mathrm{ch}_2 \right)\\
&&=\exp\!\left(-2\pi \tau_2 \left(r\Delta + \frac{1}{2r}(c_1)_+^2 -\frac{1}{2r} (c_1)_-^2  \right)+2\pi i \tau_1\left(-r\Delta +\frac{1}{2r}(c_1)_+^2 +\frac{1}{2r} (c_1)_-^2  \right)            \right) \non \\ 
&& = q^{\frac{1}{2r}(c_1)_+^2}\bar q^{r\Delta-\frac{1}{2r}(c_1)_-^2}, \non
\end{eqnarray}
where $\Delta(\gamma)$ is the discriminant defined by 
\be 
\label{eq:discriminant} 
\Delta(\gamma)=\frac{1}{r}\left(c_2-\frac{r-1}{2r}c_1^2 \right) \in \mathbb{Q}.
\ee
The sum over all $c_1\in H^2(M,\mathbb{Z})$ gives a Siegel-Narain
theta function \cite{Verlinde:1995mz}. We will however fix $c_1$ such that we can drop the sum over $c_1$. 

The action of $\CN=4$ Yang-Mills theory on a general curved surface $S$ contains also curvature terms \cite{Bachas:1999um} in addition to the terms with $F$. They contribute to the action \cite{Bachas:1999um}
$$
-2\pi i \bar \tau \frac{r\chi(S)}{24},
$$
such that the exponential in the partition function becomes
\be 
\label{eq:exponent}
q^{\frac{1}{2r}(c_1)_+^2}\bar q^{r\Delta-\frac{r\chi(S)}{24}-\frac{1}{2r}(c_1)_-^2}.
\ee
 
\section{Semi-stable sheaves on rational surfaces and change of polarization}
\label{sec:change}
We start in this section with briefly recalling necessary ingredients of
semi-stable sheaves on rational surfaces and changes of the
polarization. Let $\gamma$ be the Chern character of a coherent
sheaf: $\gamma=(r,c_1,\mathrm{ch}_2)$. The polarization $J$ of an algebraic surface $S$ is an element of the
closure of the ample cone $C(S)$ of $S$. The polarization enters in
the definition of a stability condition $\varphi_J(\gamma)$ for coherent sheaves on $S$. 
The two relevant examples of stability conditions for this paper are $\mu$-stability, $\varphi^{\mu}_J(\gamma)=\mu(\gamma)\cdot J=c_1\cdot
J/r$,  and Gieseker stability with $\varphi^{\mathrm{Gi}}_J(\gamma)=p_J(\gamma)$
where $p_J(\gamma)$ is the Hilbert polynomial of the sheaf: 
\be
p_{J}(\gamma)=J^2/2+\left( \frac{c_1(F)\cdot J}{r(F)}-\frac{K_S\cdot
    J}{2}\right) +\frac{1}{r(F)}\left(\frac{c_1(F)^2-K_S\cdot c_1(F)}{2}-c_2(F) \right)+\chi(\CO_S). \non
\ee

Most of the discussion in this article considers generating functions of so-called
virtual Poincar\'e polynomials $\CI(\gamma,w;J)$ of the moduli stack
$\mathfrak{M}_J(\gamma)$ of sheaves which are $\mu$-semi-stable with respect to
$J$ \cite{Joyce:2004}. Their formal definition is rather abstract and
involved \cite{Joyce:2004}, however their change under variations of $J$
(wall-crossing formula) as well as their generating functions take a rather simple
form. They uniquely determine the integer invariants $\Omega(\gamma,w;J)$
which conjecturally enumerate quantum BPS states. 

To explain the connection, let $\bar \Omega(\gamma,w;J)$ be the rational BPS invariant defined in terms of $\mathcal{I}(\gamma,w;J)$ by the relation \cite{Joyce:2004}:
\be
\label{eq:bOmCI}
\bar \Omega(\gamma,w;J):=\sum_{\gamma_1+\dots+\gamma_\ell=\gamma \atop p_J(\gamma_i)=p_J(\gamma)\,\,\mathrm{for}\,\,i=1,\dots,\ell} \frac{(-1)^{\ell-1}}{\ell}  \prod_{i=1}^\ell \mathcal{I}(\gamma_i,w;J) 
\ee
with inverse:
\be
\label{eq:OmI}
\mathcal{I}(\gamma,w;J)=\sum_{\gamma_1+\dots+\gamma_\ell=\gamma \atop p_J(\gamma_i)=p_J(\gamma)\,\,\mathrm{for}\,\,i=1,\dots,\ell} \frac{1}{\ell!} \prod_{i=1}^\ell \bar \Omega(\gamma_i,w;J). \non 
\ee
At generic points of the polarization, away from walls of marginal stability and boundary points, the $\bar \Omega(\gamma,w;J)$ can be further related to Laurent polynomials $P(\gamma,w;J)\in \mathbb{Z}[w,w^{-1}]$, which are symmetric under $w \leftrightarrow w^{-1}$. To this end, define $\Omega(\gamma,w;J)$ by:
\be
\label{eq:bOmtoOm}
\Omega(\gamma,w;J):=\sum_{m|\gamma}\frac{\mu(m)}{m} \, \bar \Omega(\gamma/m, -(-w)^{m};J), 
\ee
where $\mu(m)$ is the arithmetic M\"obius function. Equation (\ref{eq:bOmtoOm}) has inverse:
\be
\bar \Omega(\gamma,w;J)=\sum_{m|\gamma} \,\frac{\Omega(\gamma/m,-(-w)^m;J)}{m} \non
\ee
Then away from walls of marginal stability $\Omega(\gamma,w;J)$ takes the form:
\be
\label{eq:limitnum}
\Omega(\gamma,w;J)=\frac{P(\gamma,w;J)}{w-w^{-1}}, \non 
\ee
where $P(\gamma,w;J)$ is a Laurent polynomial symmetric under $w\leftrightarrow w^{-1}$. 
The integer BPS invariants $\Omega(\gamma;J)$ are obtained from these
by 
\be
\Omega(\gamma;J)=\lim_{w\to -1} (w-w^{-1})\,\Omega(\gamma,w;J)
\ee
 and similarly for the rational numerical invariants $\bar \Omega(\gamma;J)$.

If $\gamma$ is primitive and semi-stability implies stability than the
moduli space $\CM_J(\gamma)$ is smooth and compact and
$w^{\dim_\mathbb{C}\CM_J(\gamma)}P(\gamma,w;J)$ equals the Poincar\'e
polynomial
$\sum_{\ell=0}^{2\dim_\mathbb{C}\mathcal{M}_J(\gamma)}b_\ell(\mathcal{M}_J(\gamma))w^\ell$
of $\mathcal{M}_J(\gamma)$. If semi-stable does not imply stable, the
precise cohomological meaning of $P(\gamma,w;J)$ is not completely
clear. But following \cite{Yoshioka:1995} one expects that the Laurent polynomial gives dimensions of
intersection cohomology groups.  

To state the change of $\CI(\gamma,w;J)$ under wall-crossing, we define the function $S(\{\gamma_i\},\varphi,J,J')$ as in \cite[Definition
4.2]{Joyce:2004}.  
\begin{definition}
\label{def:S}
Let $\{\gamma_1,\gamma_2,\dots,
\gamma_\ell\}$ be a set of Chern characters with $r_i\in \mathbb{N}^*$, $i=1,\dots,\ell$. If for all $i=1,\dots, \ell-1$ we have either
\begin{enumerate}
\item[(a)] $\varphi_J(\gamma_i) \leq \varphi_J(\gamma_{i+1})$ and
  $\varphi_{J'}(\sum_{j=1}^i \gamma_j) > \varphi_{J'}(
  \sum_{j=i+1}^\ell\gamma_{j})$, or
\item[(b)] $\varphi_J(\gamma_i) > \varphi_{J}(\gamma_{i+1})$ and
  $\varphi_{J'}(\sum_{j=1}^i \gamma_j)\leq \varphi_{J'}(
  \sum_{j=i+1}^\ell\gamma_{j}),$
\end{enumerate}
then define $S(\{\gamma_i\},\varphi_J,\varphi_{J'})=(-1)^k$ where $k$ is the number of
$i=1,\dots,\ell-1$ satisfying (a). Otherwise,
$S(\{\gamma_i\},\varphi_J,\varphi_{J'})=0$. 
\end{definition}
\noindent For Gieseker stability $\varphi^{\rm Gi}_J$, the orderings $\leq$ and $>$ in Definition \ref{def:S} are to be replaced by the lexicographic ordering, 
$\preceq$ and $\succ$, respectively. In the following, we will
consider mostly $\mu$-stability and we therefore shorten notation by defining
$S(\{\gamma_i\},J,J'):=S(\{\gamma_i\},\varphi^\mu_J,\varphi^\mu_{J'})$.

Ref. \cite{Joyce:2004} shows that for surfaces whose anti-canonical
class $-K_S$ is numerically effective, the change of the invariants
$\CI(\gamma,w;J)$ under wall-crossing is expressed in 
terms of $S(\{\gamma_i\},J,J')$ as \cite[Theorem 6.21]{Joyce:2004}:

\begin{theorem}
\label{the:chCI}
Under a change of polarization $J\to J'$, the invariants
$\CI(\gamma,w;J')$ are expressed in terms of $\CI(\gamma,w;J)$ by:
\be
\label{eq:chCI}
\CI(\gamma,w;J')=\sum_{\sum_{i=1}^\ell \gamma_i=\gamma, \atop r_i \geq
  1, i=1,\dots, \ell}
S(\{\gamma_i\},J,J')\,w^{-\sum_{i<j}r_ir_j(\mu_j-\mu_i)\cdot
  K_S}\prod_{i=1}^\ell \CI(\gamma_i,w;J). 
\ee
\end{theorem}

We will restrict the computations in this article to only two rational surfaces,
namely the projective plane $\bP^2$ and its blow-up, the Hirzebruch
surface $\Sigma_1$. Let $C\cong \bP^1$ be the base curve and $f\cong \bP^1$ be the
 fibre of $\Sigma_1$, then $H_2(\Sigma_1,\mathbb{Z})=\mathbb{Z}C\oplus\mathbb{Z}f$, with
 intersection numbers $C^2=-1$, $f^2=0$ and $C\cdot f=1$. 
 The anti-canonical class $K_{\Sigma_1}$ is numerically effective and given by $-K_{\Sigma_1}=2C+3f$. We parametrize the
 closure $\overline{C(S)}$ by:
\be 
\label{eq:polarization}
J_{m,n}=m(C+f)+nf, \qquad m,n\geq 0.\non
\ee
The blow-up of $\bP^2$ be given by $\phi: \Sigma_1 \to \mathbb{P}^2$. The exceptional divisor
of $\phi$ is $C$, and the hyperplane class $H$ of
$\bP^2$ is the pullback $\phi^*(C+f)$.

\section{Generating functions for $\CI(\gamma,w;J)$}
\label{sec:genfunctions}

We now define the two generating functions
$H_{r,c_1}(\tau,z;J)$ and $h_{r,c_1}(\tau,z;J)$ with $\im(\tau) >0$ and $z\in \mathbb{C}\backslash \{ \mathrm{poles}\}$.
As usual, we let $q:=e^{2\pi i\tau}$ and $w:=e^{2\pi i z}$. Then using the expression for $\exp(-\CS(A))$ (\ref{eq:exponent}), we define the following generating function:
\begin{eqnarray}
\CZ(\tau,z;J)&=&\sum_{c_2\in H^4(S,\mathbb{Z}),\atop c_1\in H^2(S,\mathbb{Z})} \bar \Omega(\gamma,w;J)\,e^{-\CS(A)} \non \\
&=&\sum_{c_2\in H^4(S,\mathbb{Z}),\atop c_1\in H^2(S,\mathbb{Z})} \bar \Omega(\gamma,w;J)\,q^{\frac{1}{2r}(c_1)_+^2}\bar q^{r\Delta(\gamma)-\frac{r\chi(S)}{24}-\frac{1}{2r}(c_1)_-^2},
\end{eqnarray}
where $\Delta(\gamma)$ is the discriminant defined by (\ref{eq:discriminant}), $\chi(S)$ the Euler number of the complex surface $S$ and $c_{1\pm}$ are the projections defined in (\ref{eq:plusmin}). Vafa and Witten \cite{Vafa:1994tf} derived that this generating function equals the path integral of topologically twisted Yang-Mills theory with $\CN=4$ supersymmetry and gauge group $U(r)$ (for $w\to -1$). Twisting a sheaf with a line bundle $F\to \mathcal{L}\otimes F$ induces an isomorphism of moduli spaces. Since this twist leaves invariant the discriminant $\Delta$, $\CZ(\tau,z;J)$ can be written as 
\be
\CZ(\tau,z;J)=\sum_{\mu\in H^2(S;\mathbb{Z}/r\mathbb{Z})} \overline{h_{r,\mu}(z,\tau;J)}\,\Theta_{r,\mu}(\tau;J),
\ee
where $h_{r,c_1}(\tau,z;J)$ is defined by:
\be 
h_{r,c_1}(z,\tau;J):=\sum_{c_2} \bar \Omega(\gamma,w;J)\,q^{r\Delta(\gamma)-\frac{r\chi(S)}{24}},\non
\ee
and $\Theta_{r,c_1}(\tau;J)$ is the following theta function
\be
\Theta_{r,\mu}(\tau;J)=\sum_{\bs k\in \mu+ H^2(S,r\mathbb{Z})} q^{\bs k^2_+/2} q^{\bs k^2_-/2}.
\ee
The projections $\bs k_+$ and $\bs k_-$ are as defined in Equation (\ref{eq:plusmin}). With $\Theta_{r,\mu}(\tau;J)$ defined this way, $\CZ(\tau,z;J)$ typically transforms under a congruence subgroup of $SL(2,\mathbb{Z})$. We refer to \cite{Manschot:2011dj} for the details of the theta function, such that the generating function transforms under the full $SL(2,\mathbb{Z})$.

We will in the following consider the generating function of the virtual Poincar\'e polynomials $\CI(\gamma,w;J)$ rather than the rational invariants $\bar \Omega(\gamma,w;J)$. This generating function is defined by
\be
\label{eq:defH}
H_{r,c_1}(z,\tau;J):=\sum_{c_2} \CI(\gamma,w;J)\,q^{r\Delta(\gamma)-\frac{r\chi(S)}{24}}
\ee
Using (\ref{eq:OmI}), $h_{r,c_1}(z,\tau;J)$ can be expressed in terms of $H_{r,c_1}(z,\tau;J)$ and vice versa. In the following we will consider only two surfaces, the Hirzebruch surface $\Sigma_1$ and the projective plane $\bP^2$. We let $H_{r,c_1}(z,\tau;J)$ be the generating function for invariants of $\Sigma_1$ with respect to the polarization $J$, and $H_{r,c_1}(z,\tau;\bP^2)$ the generating function for $\bP^2$ which has no explicit dependence on its polarization. 

Mozgovoy \cite{Mozgovoy:2013zqx} proved using the Hall algebra of $\bP^1$ the conjecture in
\cite{Manschot:2011ym} that the generating functions $H_{r,c_1}(z,\tau;J)$ take a particularly simple form for $J=J_{0,1}=f$. One has:
\be
\label{eq:Hzf}
H_{r,c_1}(z,\tau;J_{0,1})=\left\{ \begin{array}{cl}
   H_{r}(z,\tau),
    & \mathrm{if}\,\,c_1\cdot f=0\mod r,\quad r\geq 1, \\ 0, &
    \mathrm{if}\,\,c_1\cdot f\neq 0\mod r, \quad r>1. \end{array} \right.
\ee
with
$$
H_{r}(z,\tau):=\frac{i\,(-1)^{r-1}\,\eta(\tau)^{2r-3}}{\theta_1(2z,\tau)^2\,\theta_1(4z,\tau)^2\dots\theta_1((2r-2)z,\tau)^2\,\theta_1(2rz,\tau)},
$$
with
\begin{eqnarray}
&&\eta(\tau)\quad \,\,:=q^{\frac{1}{24}}\prod_{n=1}^\infty (1-q^n),\non\\
&&\theta_1(z,\tau):=i\sum_{r\in \mathbb{Z}+\frac{1}{2}} (-1)^{r-\frac{1}{2}} q^{\frac{r^2}{2}}w^{r} \non\\
&& \qquad\quad\,\,\, =i (w^{1\over 2}-w^{-{1\over 2}})\prod_{n\geq 1}(1-q^n)(1-wq^n)(1-w^{-1}q^n). \non
\end{eqnarray}

To determine the generating functions $H_{r,c_1}(z,\tau,\bP^2)$ we use the techniques originally put forward by Yoshioka \cite{Yoshioka:1994}: change of polarization from $J_{0,1}$ to $J_{1,0}$ using Theorem \ref{the:chCI} followed by the blow-up formula. Let $\phi: \Sigma_1\to \bP^2$ be the blow-up map of a point of $\bP^2$, such that $\phi^*c_1\in H^2(\Sigma_1,\mathbb{Z})$ is the pull back of the first Chern class of sheaves on $\bP^2$. Then one has \cite{Yoshioka:1996, Gottsche:1998, Li:1999}: 
\be
\label{eq:blowup}
H_{r,c_1}(z,\tau;\bP^2)=\frac{H_{r,\phi^*c_1-kC}(z,\tau; J_{1,0})}{B_{r,k}(z,\tau)}, \qquad k\in \mathbb{Z}, 
\ee
with
\be
\non
B_{r,k}(z,\tau)=\frac{1}{\eta(\tau)^r} \sum_{\sum_{i=1}^r a_i=0 \atop a_i\in \mathbb{Z}+\frac{k}{r}}q^{-\sum_{i<j}a_ia_j}w^{\sum_{i<j}a_i-a_j}.
\ee
The blow-up formula (\ref{eq:blowup}) implies non-trivial relations between the $H_{r,\widetilde c_1}(z,\tau; J_{1,0})$, since different choices of $\widetilde c_1=\phi^*c_1-kC\in H^2(\Sigma_1,\mathbb{Z})$ can correspond to the same $c_1\in H^2(\mathbb{P}^2,\mathbb{Z})$.

The crucial step to determine $H_{r,c_1}(z,\tau;\bP^2)$ is to obtain a closed form for
$H_{r,c_1}(z,\tau;J_{1,0})$. To this end, we aim to decompose $H_{r,c_1}(z,\tau;J_{1,0})$ into more elementary building blocks.  Substitution of Eq. (\ref{eq:chCI}) in (\ref{eq:defH}) gives for $H_{r,c_1}(z,\tau;J_{1,0})$
\begin{eqnarray}
\label{eq:formgenfunction}
H_{r,c_1}(z,\tau;J_{1,0})&=&\sum_{\rm ch_2} \sum_{\sum_{i=1}^\ell \gamma_i=(r,c_1,\mathrm{ch}_2)} S(\{\gamma_i\},J_{0,1},J_{1,0})\, w^{-\sum_{j<i}r_ir_j(\mu_i-\mu_j)\cdot
  K_{\Sigma_1}}q^{r\Delta(\{\gamma_i\})-\frac{r}{6}} \non \\
&&\times\prod_{i=1}^\ell \CI(\gamma_i,w;J_{0,1}), \non
\end{eqnarray} 
where $\Delta(\{\gamma_i\})$ is the discriminant of a filtration $0\subset
F_1 \subset F_2 \subset \dots \subset F_\ell=F$ of the sheaf $F$,
whose quotients $E_i=F_i/F_{i-1}$ have Chern character
$\gamma_i$. We know from Equation (\ref{eq:Hzf}) that $H_{r,c_1}(z,\tau;J_{0,1})$ vanishes for $c_1\cdot f\neq 0 \mod r$. As a result, we can write $H_{r,c_1}(z,\tau;J_{1,0})$ with $c_1=bC-af$ as
\be
\label{eq:Hrab}
H_{r,bC-af}(z,\tau;J_{1,0})=\sum_{r_1+\dots+r_\ell=r, \, r_i\in \mathbb{N}^* }   \Psi_{(r_1,\dots,r_\ell),(a,b)}(z,\tau)\,\prod_{j=1}^\ell H_{r_j}(z,\tau), 
\ee 
where, for $r_i\in \mathbb{N}^*$ and $a, b\in \mathbb{Z}$, $\Psi_{(r_1,\dots,r_\ell),(a,b)}(z,\tau)$ is defined by
\begin{eqnarray}
\label{eq:defPsi}
&&\Psi_{(r_1,\dots,r_\ell),(a,b)}(z,\tau):=\\
&&\qquad \sum_{\sum_{i=1}^\ell (r_i,c_{1,i})=(r,bC-af)} S(\{\gamma_i\},J_{0,1},J_{1,0})\, w^{-\sum_{j<i}r_ir_j(\mu_i-\mu_j)\cdot
  K_{\Sigma_1}}q^{r\Delta(\{\gamma_i\})-\sum_{i=1}^\ell r_i\Delta(\gamma_i) },\non
\end{eqnarray}
with $\gamma_i=(r_i, c_{1,i}, {\rm ch}_{2,i})$. Note that there is no sum over the second Chern characters ${\rm ch}_{2,i}$ in $\Psi_{(r_1,\dots,r_\ell),(a,b)}(z,\tau)$ since these are captured by the functions $H_{r_i}(z,\tau)$. 
 
As defined in Equation (\ref{eq:defPsi}), $\Psi_{(r_1,\dots,r_\ell),(a,b)}(z,\tau)$ is only convergent for a finite radius of convergence of $w^{-4}$ and $w^{-4}q$. The following proposition shows that part of the sums can be carried out using geometric series, such that  $\Psi_{(r_1,\dots,r_\ell),(a,b)}(z,\tau)$ becomes an analytic function of $z\in \mathbb{C}$ and $\tau \in \mathcal{H}$ away from the poles.
\begin{proposition} 
\label{prop:HrbCaf}
The function $\Psi_{(r_1,\dots,r_\ell), (a,b)}(z,\tau)$, defined in Equation (\ref{eq:defPsi}), is given for any choice of
$r_i\in \mathbb{N}^*$, $i=1,\dots,\ell$, and $a,\,b\in \mathbb{Z}$ by
\begin{eqnarray} 
\label{eq:Psi} 
\Psi_{(r_1,\dots,r_\ell), (a,b)}(z,\tau)&=&\sum_{r_1b_1+\dots+r_\ell b_\ell=b, \atop
  b_i\in \mathbb{Z}} \frac{w^{\sum_{j<i} r_ir_j(b_i-b_j) +2(r_i+r_{i-1})  \left\{ \frac{a}{r} \sum_{k=i}^\ell r_k \right\} }}{\prod_{i=2}^{\ell} \left(1-w^{2(r_{i}+r_{i-1})}q^{b_{i-1}-b_{i}}\right)} \non \\
&& \times\, q^{\sum_{i=1}^\ell \frac{r_i(r-r_i)}{2r}b_i^2 - \frac{1}{r} \sum_{i<j} r_ir_j b_ib_j + \sum_{i=2}^\ell (b_{i-1}-b_{i})\left\{ \frac{a}{r} \sum_{k=i}^\ell r_k \right\}  } , \non
\end{eqnarray} 
where $\{\lambda\}=\lambda-\lfloor \lambda \rfloor$ is the rational part of $\lambda$.\footnote{We hope that the double usage of $\{\}$, for denoting either a rational part or a set, will not lead to confusion.}
\end{proposition}
 
\begin{proof}
The discriminant $\Delta(\{\gamma_i\})$ (\ref{eq:discriminant}) of a filtration is expressed in
terms of Chern characters $\gamma_i=(r_i,c_{1,i},\mathrm{ch}_{2,i})$ of the quotients by:
\be
r\Delta(\{\gamma_i\})=\sum_{i=1}^\ell r_i \Delta(E_i)-\sum_{i=2}^\ell
\frac{1}{2r_i} \frac{1}{\sum_{j=1}^i r_j\sum_{k=1}^{i-1}r_k} \left(
  \sum_{j=1}^{i-1}r_i c_{1,j}-r_j c_{1,i} \right)^2. \non
\ee
Substitution of this expression in Equation (\ref{eq:defPsi}) gives
\begin{eqnarray} 
\label{eq:form2genfunction}
\Psi_{(r_1,\dots,r_\ell), (a,b)}(z,\tau)&=&\sum_{\sum_{i=1}^\ell
  (r_i,c_{1,i})=(r,bC-af), \atop r_i\in \mathbb{N}^*  ,\, i=1,\dots, \ell} S(\{\gamma_i\},J_{0,1},J_{1,0})\, w^{-\sum_{j<i}r_ir_j(\mu_i-\mu_j)\cdot
  K_S} \\ 
&&\times \, q^{-\sum_{i=2}^\ell
\frac{1}{2r_i} \frac{1}{\sum_{j=1}^i r_j\sum_{k=1}^{i-1}r_k} \left(
  \sum_{j=1}^{i-1}r_i c_{1,j}-r_j c_{1,i} \right)^2} \non
\end{eqnarray} 
We parametrize the first Chern classes as $c_1=bC-af$ and $c_{1,i}=r_ib_iC-a_if$, such that the sets $\{a_i\}$ and $\{b_i\}$ have to satisfy
$\sum_{i=1}^\ell a_i=a$ and $\sum_{i=1}^\ell r_ib_i=b$. We continue
with bringing $\Psi_{(r_1,\dots,r_\ell), (a,b)}(z,\tau)$ to a form which allows
to carry out the sums over $a_i$, $i=2,\dots,\ell$ of the different contributions. 
First, after substitution of $c_{1,i}=r_ib_iC-a_i f$ in
$\Delta(\{\gamma_i\})$ and $a_1=a-\sum_{i=2}^{\ell}a_i$ in $\Delta(\{\gamma_i\})$, one obtains: 
\begin{eqnarray}
\label{eq:rDf}
r\Delta(\{\gamma_i\})&=&\sum_{i=1}^\ell r_i\Delta(\gamma_i)+ \sum_{i=2}^\ell \frac{r_i}{2 \sum_{j=1}^i r_j\sum_{k=1}^{i-1}r_k} \left(\sum_{j=1}^{i-1} r_j(b_j-b_i) \right)^2 \non \\
&& + \sum_{i=2}^\ell \frac{1}{\sum_{j=1}^i r_j\sum_{k=1}^{i-1}r_k} \left(\sum_{j=1}^{i-1} r_j(b_i-b_j) \right) \left(\sum_{k=i+1}^\ell r_i a_k + \sum_{k=1}^{i} r_ka_i-r_ia \right).
\end{eqnarray}
We now make the following change of variables:
\be
\label{eq:subsas}
a_i = s_i-s_{i+1}, \quad i=2,\dots,\ell-1, \qquad a_\ell = s_\ell. 
\ee
This transforms the second line of (\ref{eq:rDf}) to:
\be
 \sum_{i=2}^\ell \left(\sum_{j=1}^{i-1} r_j(b_i-b_j) \right)\left( \frac{s_i}{\sum_{j=1}^{i-1}r_j}-\frac{s_{i+1}}{\sum_{j=1}^{i}r_j}-\frac{r_i a}{\sum_{j=1}^i r_j\sum_{k=1}^{i-1}r_k}\right),\non
\ee 
which can be further simplified to:
$$
\sum_{i=2}^\ell (b_i-b_{i-1})s_i -a\sum_{i=2}^\ell \frac{\sum_{m=1}^{i-1}r_ir_m(b_i-b_m)}{\sum_{j=1}^i r_j\sum_{k=1}^{i-1}r_k}.
$$

The exponent of $w$ in Eq. (\ref{eq:defPsi}) is easily evaluated in terms of
$a_i$ and $b_i$:
$$-\sum_{j<i}r_ir_j(\mu_i-\mu_j)\cdot
  K_{\Sigma_1}=\sum_{j<i} r_ir_j(b_i-b_j)-2(r_ja_i-r_ia_j).$$
Replacing $a_1$ as before this becomes:
\be
\label{eq:wpower}
\sum_{j<i} r_ir_j(b_i-b_j) -2\left( 2\sum_{1<j<i}r_ja_i + (r_1+r_j)a_j-(r-r_1)a\right). \non
\ee 
The substitution (\ref{eq:subsas}) then gives:
$$
2(r-r_1)a+\sum_{j<i} r_ir_j(b_i-b_j)-2\sum_{i=2}^\ell(r_{i-1}+r_i)s_i.
$$

Now we come to the third term of the summand in (\ref{eq:form2genfunction}): $S(\{\gamma_i\},J,J')$. Interestingly, this can be written as a
product of differences of signs, which are familiar from the
literature on indefinite theta functions \cite{Gottsche:1996, Zwegers:2000}. To this end, define
\be 
\sgn(x):=\left\{ \begin{array}{rl} 1, & \mathrm{if}\,\,x>0, \\ 0, &
    \mathrm{if}\,\, x=0, \\ -1, & \mathrm{if}\,\, x<0.  \end{array}\right. \non
\ee
Then we can write $S(\{\gamma_i\},J,J')$ as
\begin{eqnarray*}
&&S(\{\gamma_i\},J,J')=\\
&&\qquad \frac{1}{2^{\ell-1}}\prod_{i=2}^{\ell}
\left(\sgn(\varphi_J(\gamma_{i-1})-\varphi_J(\gamma_{i})-v_1)-
  \sgn\!\left(\varphi_{J'}\!\left(\sum_{j=1}^{i-1}\gamma_j\right)-\varphi_{J'}\!\left(\sum_{j=i}^\ell\gamma_{j}\right)-v_1\right)\right), \non
\end{eqnarray*}
where $v_1>0$ is a sufficiently small positive constant such that $0<v_1<|\varphi_J(\gamma)-\varphi_J(\gamma')|$ for each $\gamma$, $\gamma'$ such that  $\varphi_J(\gamma)-\varphi_J(\gamma')\neq 0$. Specializing to $\varphi_J=\varphi^\mu_J$ and substitution of $c_{1,i}=b_ir_iC-a_i f$ gives:
\be
S(\{\gamma_i\};J_{0,1},J_{1,0})=\frac{(-1)^{\ell-1}}{2^{\ell-1}}\prod_{i=2}^{\ell}\left(
  \sgn(b_{i}-b_{i-1}+v_2)+\sgn\!\left(\sum_{j=1}^{i-1}\sum_{k=i}^\ell
    a_kr_j - a_jr_k-v_2\right) \right). \non
\ee
where $0<v_2<1$. Making again the substitution for $a_1$ brings the argument in the second sign to the
form:
$$
\sum_{j=1}^i\sum_{k=i+1}^\ell a_kr_j-a_jr_k=r\sum_{k=i+1}^\ell a_k-a\sum_{k=i+1}^\ell r_k.
$$
With the substitution (\ref{eq:subsas}), this simplifies to:
\be
S(\{\gamma_i\}; J_{0,1}, J_{1,0})=\frac{(-1)^{\ell-1}}{2^{\ell-1}}\prod_{i=2}^{\ell}\left( \sgn(b_{i}-b_{i-1}+v_2)+\sgn\!\left(r s_{i}-a \sum_{k=i}^\ell r_k -v_2 \right)\right), \non
\ee 
We now observe that the sum over the $s_i$'s are simply geometric sums and can be carried out if $z$ is such that $|w^{-4}|<1$ and $|w^4 q|<1$. This brings $\Psi_{(r_1,\dots,r_\ell), (a,b)}(z,\tau)$ to the following form:  
\begin{eqnarray*}
\label{eq:Psi2}
\Psi_{(r_1,\dots,r_\ell), (a,b)}(z,\tau)&:=&\sum_{r_1b_1+\dots+r_\ell b_\ell=b, \atop
  b_i\in \mathbb{Z}} (-1)^{\ell-1}\frac{w^{2(r-r_1)a+\sum_{j<i} r_ir_j(b_i-b_j) -2(r_i+r_{i-1}) \left( 1+ \left\lfloor \frac{a}{r} \sum_{k=i}^\ell r_k \right\rfloor\right)}}{\prod_{i=2}^{\ell} \left(1-w^{-2(r_{i}+r_{i-1})}q^{b_i-b_{i-1}}\right)} \non \\
&& \times\, q^{\sum_{i=2}^\ell \frac{r_i}{2 \sum_{j=1}^i r_j\sum_{k=1}^{i-1}r_k} \left(\sum_{j=1}^{i-1} r_j(b_j-b_i) \right)^2 }  \\
&&\times\, q^{\sum_{i=2}^\ell (b_i-b_{i-1})\left( 1+ \left\lfloor \frac{a}{r} \sum_{k=i}^\ell r_k \right\rfloor\right)  -a \frac{\sum_{j=1}^{i-1}r_ir_j(b_i-b_j)}{\sum_{k=1}^i r_k\sum_{m=1}^{i-1}r_m}}. \non 
\end{eqnarray*}
which can immediately be analytically continued to $z\in \mathbb{C}\backslash \{\mathrm{poles}\}$. To bring it to the simpler form of the Proposition, one proves easily with induction on $\ell$ that: 
\begin{eqnarray}   
&& \sum_{i=2}^\ell (r_i+r_{i-1})\sum_{k=i}^\ell r_k=(r-r_1)r,\non \\ 
&& r\sum_{i=2}^\ell \sum_{j=1}^{i-1} \frac{r_ir_j(b_i-b_j)}{\sum_{k=1}^ir_k\sum_{m=1}^{i-1}r_m}=b-b_1r=\sum_{i=2}^\ell(b_i-b_{i-1})\sum_{k=i}^\ell r_k, \non \\
&& \sum_{i=2}^\ell \frac{r_i}{2 \sum_{j=1}^i r_j\sum_{k=1}^{i-1}r_k} \left(\sum_{j=1}^{i-1} r_j(b_j-b_i) \right)^2=\sum_{i=1}^\ell \sum_{j\neq \ell} \frac{r_ir_j}{2r} \left(b_i^2-2b_ib_j \right) \non
\end{eqnarray}
Substitution of these expressions in Equation (\ref{eq:Psi2}) gives the proposition. Note that it is manifestly invariant under shifts of $(a,b)\to (a,b)+r(k_1,k_2)$ with $k_1,k_2\in \mathbb{Z}$. 
\end{proof}

We note that Proposition \ref{prop:HrbCaf} gives already for $r=3$ much simpler expressions than those in \cite{Kool:2009, weist:2009, Manschot:2010nc, Manschot:2011ym}. The simpler expression allows for a rather quick determination of the invariants. We have verified that the first coefficients of $H_{r,c_1}(z,\tau;\bP^2)$ reproduce those in the previous literature in all known cases. 

We finish this section with an example. We compute the integer invariants $\Omega(\gamma,w;\bP^2)$ of 
sheaves on $\bP^2$ with $(r,c_1)=(4,2H)$. Eq. (\ref{eq:bOmCI})
shows that we need to determine both $H_{2,H}(z,\tau;\bP^2)$ and
$H_{4,2H}(z,\tau;\bP^2)$. For $H_{2,H}(z,\tau;\bP^2)$ we determine first $H_{2, C+f}(z,\tau;J_{1,0})$. The only contributing term of the sum
$\Sigma_{r_1+\dots+r_\ell=2}$ with solutions to $r_1b_1+r_2b_2=1$ in Eq. (\ref{eq:Hrab}) is $(r_1,r_2)=(1,1)$. This gives immediately the result of Yoshioka \cite{Yoshioka:1994}:
\be
\label{eq:2H}
H_{2,H}(z,\tau;\bP^2)=\frac{H_1(z,\tau)^2}{B_{2,0}(z,\tau)}\sum_{k\in
\mathbb{Z}} \frac{w^{-2k+1}q^{k^2+2k+\frac{3}{4}}}{1-w^{4}q^{2k+1}}.
\ee
Alternatively, one can determine $H_{2,H}(z,\tau;\bP^2)$ starting from $H_{2, f}(z,\tau;J_{1,0})$. This gives
\be
\label{eq:2H2}
H_{2,H}(z,\tau;\bP^2)=\frac{H_2(z,\tau)}{B_{2,1}(z,\tau)}+\frac{H_1(z,\tau)^2}{B_{2,1}(z,\tau)}\sum_{k\in
\mathbb{Z}} \frac{w^{-2k+2}q^{k^2+k}}{1-w^{4}q^{2k}}.
\ee 
Section \ref{sec:appell} will explain how the equality of (\ref{eq:2H}) and (\ref{eq:2H2}) follows from a known relation for the classical Appell function.

For $H_{4,2H}(z,\tau;\bP^2)$ the contributing terms in the sum 
$\Sigma_{r_1+\dots+r_\ell=r}$ are $(r_1,\dots,r_\ell)=(3,1)$, $(1,3)$,
$(2,2)$, $(2,1,1)$, $(1,2,1)$, $(1,1,2)$ and $(1,1,1,1)$. The
different $\Psi_{(r_1,\dots,r_\ell),(a,b)}(z,\tau)$ with $(a,b)=(-2,2)$ are given by:
\begin{eqnarray} 
&&\Psi_{(3,1),(a,b)}(z,\tau)=\Psi_{(1,3),(a,b)}= \sum_{k\in\mathbb{Z}} \frac{w^{-12k+10}q^{6k^2-4k+\frac{1}{2}}}{1-w^{8}q^{4k-2}},\non \\
&&\Psi_{(2,2),(a,b)}(z,\tau)= \sum_{k\in\mathbb{Z}} \frac{w^{-8k-4}q^{2k^2+4k+\frac{1}{2}}}{1-w^{8}q^{2k+1}},\non \\
&&\Psi_{(2,1,1),(a,b)}(z,\tau)=\Psi_{(1,1,2),(a,b)}(z,\tau)=\sum_{k_1,k_2\in\mathbb{Z}}\frac{w^{-10k_1-2k_2+8}q^{3k_1^2+2k_1k_2+k_2^2-3k_1-k_2+\frac{1}{2}}}{(1-w^{6}q^{k_1-k_2})(1-w^{4}q^{2k_1+2k_2-2})}, \non \\
&&\Psi_{(1,2,1),(a,b)}(z,\tau)= \sum_{k_1,k_2\in\mathbb{Z}} \frac{w^{-6k_1-6k_2+10}q^{k_1^2+2k_1k_2+3k_2^2-4k_2-1}}{(1-w^{6}q^{k_1-k_2})(1-w^{4}q^{2k_1+2k_2-2})},\non \\
&&\Psi_{(1,1,1,1),(a,b)}(z,\tau)= \sum_{k_1,k_2,k_3\in\mathbb{Z}} \frac{w^{-6k_1-4k_2-2k_3+10}q^{k_1^2+k_2^2+k_3^2+k_1k_2 +k_1k_3+k_2k_3-k_1-2k_2-k_3+\frac{1}{2}}}{(1-w^{4}q^{k_1-k_2})(1-w^{4}q^{k_2-k_3})(1-w^{4}q^{k_1+k_2+2k_3-2})}.\non 
\end{eqnarray}
Summing up these functions as prescribed by Proposition \ref{prop:HrbCaf} one determines $H_{4,2C+2f}(z,\tau;J_{1,0})$. After application of the blow-up formula and using the formulas in Section \ref{sec:genfunctions}, one obtains the generating function of Poincar\'e polynomials $P(\gamma,w)$ for Gieseker semi-stable sheaves on $\bP^2$ with $\gamma=(4,2H,\mathrm{ch}_2)$:
$$
\frac{H_{4,2C+2f}(z,\tau;J_{1,0})}{B_{4,0}(z,\tau)}-\frac{1}{2}H_{2,H}(z,\tau;\bP^2)^2+\frac{1}{2}H_{2,H}(2z,2\tau;\bP^2). 
$$
Poincar\'e polynomials are listed for the first few values of $c_2$ in Table 1. Specializing to Euler numbers, one finds for $c_2=8,9,10, \dots$ the numbers $93726, 505942, 2411826,\dots$.

\begin{table}[h!]
\begin{tabular}{lrrrrrrrrrrrrrrrrr}
$c_2$ & $b_0$ & $b_2$ & $b_4$ & $b_6$ & $b_8$ & $b_{10}$ & $b_{12}$ & $b_{14}$
& $b_{16}$ & $b_{18}$ & $b_{20}$ & $b_{22}$ & $b_{24}$ & $b_{26}$ & $b_{28}$  &  $\chi$ \\
\hline
4 & 1 & 1 & 1 &  &  &  & & & & & & & & & & 6  \\
5 & 1 & 2 & 6 & 10 & 17 & 21 & 24 &  &  & & & & & & & 162 \\
6 & 1 & 2 & 6 & 13 & 27 & 49 & 84 & 126 & 173 & 211 & 231 & & & & & 1846 \\
7 & 1 & 2 & 6 & 13 & 29 & 55 & 107 & 185 & 315 & 493 & 736 & 1008 & 1290 & 1509 & 1634 & 14766 \\
\end{tabular}
\caption{The Betti numbers $b_n$ (with $n\leq
  \dim_\mathbb{C} \mathcal{M}$) and the Euler number $\chi$ of the moduli spaces of semi-stable sheaves
  on $\bP^2$ with $r=4$, $c_1=2C+2f$, and $4\leq c_2\leq 7$.}  
\label{tab:betti30}
\end{table}

\section{Appell functions}
\label{sec:appell}
This section explains that the generating functions $\Psi_{(r_1,\dots,r_\ell),(a,b)}$ can be viewed as major generalization of the classical Appell function. We first review the definition and main properties of the classical Appell function. In Subsection \ref{sec:genappell} we introduce the generalized Appell functions with signature $(n_+,n_-)$. Subsection \ref{sec:example} illustrates that the functions $\Psi_{(r_1,\dots,r_\ell), (a,b)}(z,\tau)$ of Section \ref{sec:genfunctions} are specializations of Appell functions with signature $(\ell-1,\ell-1)$, and discusses a few consequences.

\subsection{The classical Appell function} 
\label{sec:classAppell}
The classical Appell function is defined as \cite{Appell:1886}:
\be
\label{eq:appell}
A(u,v,\tau):=e^{\pi i u}\sum_{n\in \mathbb{Z}} \frac{(-1)^nq^{n(n+1)/2}e^{2\pi i nv}}{1-e^{2\pi i u}q^n}
\ee
which is a meromorphic function of $u\in \mathbb{C}$ with pole for $u\in \mathbb{Z}\tau+\mathbb{Z}$, and holomorphic in $\tau\in
\mathcal{H}$ and $v\in \mathbb{C}$. It is well-known that the transformation properties of $A(u,v,\tau)$ are not exactly those of a modular or Jacobi
form \cite{Semikhatov:2003uc, Zwegers:2000}. However, define the ``completed'' Appell function as:
$$
\widehat A(u,v,\tau):=A(u,v,\tau)+\frac{i}{2}\theta_1(v,\tau)\,R(u-v,\tau)
$$
with
$$
R(u,\tau)=\sum_{n\in \mathbb{Z}+\frac{1}{2}} \left(\,\sgn(n) - E\!\left((n+\mathrm{Im}(u)/y)\sqrt{2y}
  \right) (-1)^{n-\frac{1}{2}}e^{-2\pi i u n}q^{-n^2/2}\right)
$$
and $E(x)=2\int_{0}^x e^{-\pi u^2}du$. Then $\widehat A(u,v,\tau)$ satisfies the following properties \cite{Zwegers:2010}:
\begin{enumerate}
\item Modular transformations: for $\left(\begin{array}{cc} a & b \\ c & d \end{array}\right) \in SL(2,\mathbb{Z})$ 
\be
\widehat A\!\left(\frac{u}{c\tau+d},\frac{v}{c\tau+d},\frac{a\tau+b}{c\tau+d}\right)=(c\tau+d)\, e^{\pi i c(-u^2+2uv)/(c\tau+d)}\widehat A(u,v,\tau). \non
\ee
\item Elliptic transformations: for $k,l,m,n\in \mathbb{Z}$
\be
\widehat A(u+k\tau+l,v+m\tau+n;\tau)=(-1)^{k+m}\,e^{2\pi i (k-m)u-2\pi i k v } q^{k^2/2-km}\widehat{A}(u,v;\tau). \non
\ee 
\item Periodicity relation:
$$
\theta_1(v,\tau)\, A(u+z,v+z,\tau) - \theta_1(v+z,\tau)\, A(u,v,\tau)=\frac{\eta(\tau)^3\,\theta_1(u+v+z,\tau)\,\theta_1(z,\tau)}{\theta_1(u,\tau)\,\theta_1(u+z,\tau)}
$$
\item They can be seen as coefficients of a meromorphic Jacobi form:
$$
\sum_{m\in \mathbb{Z}} A(u+m\tau,v,\tau)\, e^{2\pi i m (z-\frac{1}{2}\tau)-\pi i u}=\frac{ \eta(\tau)^3\,\theta_1(u+z,\tau)\,\theta_1(v-z)}{i\,\theta_1(u,\tau)\,\theta_1(z,\tau)}
$$
\end{enumerate}

\subsection{Appell functions with signature $(n_+,n_-)$}
\label{sec:genappell}
The Appell function and its generalizations (the higher level Appell functions $A_\ell$ \cite{Semikhatov:2003uc} and the multivariable Appell functions $A_Q$ \cite{Kac:2001, Zwegers:2010}) have appeared at various places in the mathematics and mathematical physics literature. As will be explained in more detail in the next subsection, the functions $\Psi_{(r_1,\dots,r_\ell), (a,b)}(z,\tau)$ for $\ell>2$ motivate the introduction of a further generalization of Appell functions. These generalized Appell functions are characterized by their signature $(n_+,n_-)$. They depend on an $n_+$-dimensional lattice $\Lambda \cong \mathbb{Z}^{n_+}$ with positive definite quadratic form $Q({\bs k})={\bs k}^\mathrm{T} {\bf Q}{\bs k}$. The scalar product ${\bs k}\cdot {\bs m}$ denotes as usual $\sum_{i=1}^{n_+} k_i\,m_i$. We have furthermore an $n_+$ by $n_-$ matrix ${\bf M}$ such that the  determinant of
\be
\widetilde {\bf Q}= \left(\begin{array}{cc} {\bf Q} & {\bf M}^\mathrm{T} \\ {\bf M} & {\bf 0}\end{array} \right) \nonumber
\ee 
does not vanish. The ``signature'' of the Appell function can thus be seen as the signature $(n_+,n_-)$ of the above matrix. The column vectors of $\bf M$ are denoted by ${\bs m}_i\in \Lambda^*$, $i=1,\dots,n_-$. In addition, we have a vector ${\bs m}_0\in \Lambda^* \times \mathbb{Q}$, two complex vectors ${\bs u}=(u_1,\dots,u_{n_-})\in \mathbb{C}^{n_-}$ and ${\bs v}=(v_1,\dots,v_{n_+})\in \Lambda^* \otimes \mathbb{C}\simeq \mathbb{C}^{n_+}$, and a constant $R\in \mathbb{Q}$. In terms of this data we define an Appell function of signature $(n_+,n_-)$ as a function of the form:
\be
\label{eq:genAppell} 
A_{Q,\{{\bs m}_i\}}({\bs u},{\bs v},\tau):=e^{2\pi i {\bs m}_0\cdot {\bs u}}\sum_{{\bs
    k}\in \Lambda} \frac{q^{\frac{1}{2}Q({\bs k})+R}e^{2\pi i {\bs
      v}\cdot {\bs k}}}{\prod_{j=1}^{n_-} (1-q^{{\bs m}_j\cdot {\bs
      k}}e^{2\pi i u_j})}.
\ee 
Note that expanding the denominators as a geometric sum will bring $A_{Q,\{{\bs m}_i\}}({\bs u},{\bs v},\tau)$ to the form of an indefinite theta function of a lattice with the quadratic form $\widetilde {\bf Q}$ defined above. Appell functions with signature $(1,1)$ are the classical Appell
functions, possibly of higher level. Appell functions of signature $(n_+,1)$ with $n_+\geq
2$ are the multi-variable Appell functions studied in
\cite{Zwegers:2010}. To my knowledge, functions $A_{Q,\{{\bs m}_i\}}({\bs u},{\bs v},\tau)$ with $n_->1$ have not appeared earlier in the literature.\footnote{After finishing this note, I became aware that functions very similar to (\ref{eq:genAppell}) are described by Kac and Wakimoto in the context of Lie superalgebras \cite[Equation (0.13)]{Kac:2013}. An important difference is that for Lie superalgebras, $\{{\bs m}_i\}$ is a set of pairwise orthogonal vectors, whereas this is typically not the case for the functions of interest in this paper.} 

Analogues of all four properties of the classical Appell function listed above are expected to exist for the Appell functions with general signature. After addition of a suitable completion, the generalized Appell functions are expected to transform as a multivariable Jacobi form with weight $(n_++n_-)/2$ modular form. The modular properties will also fix as usual the values $R$ and ${\bs m}_0$. The analogue of the fourth property is most easily established. We have

\begin{proposition}
Let ${\boldsymbol z}=(z_1,\dots,z_{n_-})$ be a complex vector of length $n_-$
\be
\sum_{ {\boldsymbol l}\in \mathbb{Z}^{n_-} } A_{Q,\{{\bs m}_i\}}({\bs u}+{\boldsymbol l}\tau,{\bs v},\tau) e^{2\pi i {\boldsymbol l} \cdot ({\boldsymbol z}-{\bs m}_0 \tau ) -2\pi i {\bs m_0}u}=\Theta_Q({\bs v}-{\bf M}{\bs z},\tau)\prod_{j=1}^{n_-}\left( \frac{\eta(\tau)^3\,\theta_1(u_j+z_j,\tau)}{i\,\theta_1(u_j,\tau)\,\theta_1(z_j,\tau)} \right),\non
\ee
%
where $\Theta_Q({\bs v},\tau)$ is a theta function for the lattice with quadratic form $Q$:
$$
\Theta_Q({\bs v},\tau)=\sum_{{\bs k}\in \Lambda} q^{\frac{1}{2}Q({\bs k}^2)+R}e^{2\pi i {\bs v}\cdot {\bs k}}.
$$
\end{proposition}  

\begin{proof}
The proof follows almost immediately from the change of the summation variables ${\bs l}\to {\bs l}-{\bf M}\cdot {\bs k}$, and application of the identity:
 $$
 \sum_{m\in \mathbb{Z}}\frac{e^{2\pi i mz}}{1-e^{2\pi i u}q^m}=\frac{\eta(\tau)^3\,\theta_1(u+z,\tau)}{\theta_1(u,\tau)\,\theta_1(z,\tau)}.
 $$
\end{proof}

\subsection{Generating functions $H_{r,c_1}(z,\tau)$ and Appell functions of signature $(\ell,\ell)$}
\label{sec:example}
In this subsection, we relate the generating functions $H_{r,c_1}(z,\tau)$ and the Appell functions with a general signature, and discuss consequences based on the blow-up formula. Clearly, the functions
$\Psi_{(r_1,\dots,r_\ell),(a,b)}(z,\tau)$ in Section \ref{sec:genfunctions}  are specializations of a generalized Appell function with signature $(\ell-1,\ell-1)$. If $r_i=1, \forall i=1,\dots,\ell$, the corresponding $(\ell-1)$-dimensional quadratic form corresponds to the one of the $\operatorname{A}_{\ell-1}$ root lattice. We define the quadratic form $\mathrm{Q}_{r}$ of the $\operatorname{A}_r$ root lattice by
\be
\mathrm{Q}_r:=\left(\begin{array}{ccccc} 2 & 1 & 1 & $\dots$ & 1 \\ 1 & 2 & 1 & $\dots$ & 1 \\ $\vdots$ & $\vdots$ & $\vdots$ & $\vdots$ & $\vdots$ \\ 1  & 1 & 1 & $\dots$ & 2\end{array}\right). 
\ee

For $r=2$, explicit expressions for the $H_{2,H}(z,\tau;\mathbb{P}^2)$ are given in Equations (\ref{eq:2H}) and (\ref{eq:2H2}). The $\Psi_{(1,1),(a,b)}$ are in this case specializations of the classical Appell function. For example, the Appell function in Equation (\ref{eq:2H}) can be written as $A(4z+\tau,-2z+\tau+\frac{1}{2},2\tau)$ (up to a term of the form $q^\alpha w^\beta$). Writing also the Appell function in Equation (\ref{eq:2H2}) in terms of a specialization of $A(u,v,\tau)$, one can easily prove the equality of (\ref{eq:2H}) and (\ref{eq:2H2}) using the periodicity relation of $A(u,v,\tau)$ given in Section \ref{sec:classAppell}.

For $r=3$, the explicit expressions for the $H_{3,c_1}(z,\tau;J_{1,0})$ are listed in the appendix. One verifies straightforwardly that the $\Psi_{(1,1,1),(a,b)}$ are specializations of the following $\operatorname{A}_2$ Appell function
$$
A_{Q_{2}}\!({\bs u},{\bs v},\tau)=\sum_{k_1,k_2\in \mathbb{Z}}\frac{q^{k_1^2+k_2^2+k_1k_2}e^{2\pi i v_1(2k_1+k_2)+2\pi i v_2(k_2-k_1)}}{(1-e^{2\pi i u_1}q^{2k_1+k_2})(1-e^{2\pi i u_2}q^{k_2-k_1})}, 
$$
with ${\bs u}=(u_1,u_2)$, and ${\bs v}=(v_1,v_2)$.
The blow-up formula (\ref{eq:blowup}) gives in this case the following relations among the $H_{3,c_1}(z,\tau;J_{1,0})$:
\begin{eqnarray} 
\label{eq:blowup3}
\frac{H_{3,f}(z,\tau;J_{1,0})}{B_{3,1}(z,\tau)}&=&\frac{H_{3,-C+f}(z,\tau;J_{1,0})}{B_{3,1}(z,\tau)}=\frac{H_{3,C+f}(z,\tau;J_{1,0})}{B_{3,0}(z,\tau)},\\
\frac{H_{3,0}(z,\tau;J_{1,0})}{B_{3,0}(z,\tau)}&=&\frac{H_{3,C}(z,\tau;J_{1,0})}{B_{3,1}(z,\tau)}, \non
\end{eqnarray}
which clearly imply very intricate relations among the $\Psi_{(1,1,1),(a,b)}$.\footnote{Using the new expressions for $H_{3,c_1}(z,\tau)$, the identities (\ref{eq:blowup3}) are proven in Reference \cite{Bringmann:2015} by making only use of analytic properties of the $q$-series. The first equality sign on the first line of (\ref{eq:blowup3}) corresponds to \cite[Theorem 1.1]{Bringmann:2015}, the second equality sign to \cite[Theorem 1.2]{Bringmann:2015}, and the equality on the second line to \cite[Theorem 1.3]{Bringmann:2015}. The proofs confirm that, at least for $r\leq 3$, the computation of $H_{r,c_1}(z,\tau;\mathbb{P}^2)$ using the functions $H_{r,\widetilde c_1}(z,\tau;J_{0,1})$, wall-crossing and blow-up formula, lead to identical generating functions, independent of the choice of $\widetilde c_1=\phi^*c_1-kC$.}

The identities (\ref{eq:blowup3}) also provide non-trivial information about the zeros and poles of $H_{3,c_1}(z,\tau)$. For example, one can show using techniques from \cite{Eichler:1985} that the function $B_{3,1}(z,\tau)$ has zeroes at torsion points: $z_0=\frac{n}{3}$, $n=1,2\mod 3$, whereas $B_{3,0}(z,\tau)$ has zeros for  $z_0=\pm ({1\over 2} \tau+{1\over 4}+\nu(\tau))$ where
 \begin{eqnarray} 
\nu(\tau)&=&-2\pi\int_{\tau}^{i\infty}
   (t-\tau) F(t)dt\\
&=&  \frac{q^{1\over 2}}{2\pi}\left(1-{11\over 6}q+{243 \over 40}q^2+\dots\right), \non
 \end{eqnarray}
with $F(\tau)=\frac{\eta(\tau)^{12}\,\eta(3\tau)^{6}}{(q^{1/2}\,\prod_{\lambda,\mu=0,1} \eta(\tau)^3\,B_{3,0}(\tau,\lambda\tau/4+\mu/4))^{3/2}}$. It is not hard to verify that both $H_{3,f}(z,\tau;J_{1,0})$ and $H_{3,-C+f}(z,\tau;J_{1,0})$ do not have poles at points where $B_{3,0}(z,\tau)$ vanishes. Therefore, $H_{3,C+f}(z,\tau;J_{1,0})$ must necessarily vanish at the same points where $B_{3,0}(z,\tau)$ vanishes. One can verify order by order that this is indeed the case.

As a final illustration, we express $\Psi_{(1,1,1,1),(-2,2)}(z,\tau)$ as a generalized Appell function of the form (\ref{eq:genAppell}). The corresponding quadratic form is ${\rm Q}_{3}$, and the vectors ${\bs m}_i$, $i=0,\dots,3$ are given by: 
$$
{\bs m}_0=\frac{1}{2}(1,0,0),\quad {\bs m}_1=(1,-1,0), \quad {\bs m}_2=(0,1,-1),\quad {\bs m}_3=(1,1,2) ,
$$
${\bs u}=-4z(1,1,1)+2\tau(0,0,1)$, ${\bs v}=(6,4,2)z+(3,3,2)\tau$, and $R=\frac{5}{2}$.

\appendix

\section{Explicit expressions for $H_{3,c_1}(z,\tau;J_{1,0})$}
 
We list here explicit expressions for the functions $H_{3,c_1}(z,\tau;J_{1,0})$. For any $J\in C(S)$, the functions $H_{3,c_1}(z,\tau;J)$ are well-known to satisfy 
$$H_{3,c_1}(z,\tau;J)=H_{3,-c_1}(z,\tau;J)=H_{3,c_1+k}(z,\tau;J),$$
with $k\in H^2(\Sigma_1,3\mathbb{Z})$. As a result, there are only five different functions $H_{3,c_1}(z,\tau;J_{1,0})$. Using the results of Section \ref{sec:genfunctions}, one obtains for these
\begin{eqnarray}
H_{3,f}(z,\tau;J_{1,0})&=&H_{1}(z,\tau) H_{2}(z,\tau)\left( \sum_{k\in \mathbb{Z}}
  \frac{w^{-6k+4}q^{3k^2+2k}}{1-w^{6}q^{3k}}+ \sum_{k\in \mathbb{Z}}
  \frac{w^{-6k+2}q^{3k^2+k}}{1-w^{6}q^{3k}} \right) \non \\
&&+H_{1}(z,\tau)^3\sum_{k_1,k_2\in \mathbb{Z}}\frac{w^{-2(k_1+2k_2-2)}q^{k_1^2+k_2^2+k_1k_2+k_1+k_2}}{(1-w^{4}q^{2k_1+k_2})(1-w^{4}q^{k_2-k_1})}+H_{3}(z,\tau), \non
\end{eqnarray}
\begin{eqnarray}  
H_{3,-C+f}(z,\tau;J_{1,0})&=&H_{1}(z,\tau) H_{2}(z,\tau) \left( \sum_{k\in \mathbb{Z}}
\frac{w^{-6k+2}q^{3k^2+4k+1}}{1-w^{6}q^{3k+1}}+ \sum_{k\in \mathbb{Z}} \frac{w^{-6k+4}q^{3k^2-k}}{1-w^{6}q^{3k-1}}\right) \non \\
&&+ \, H_{1}(z,\tau)^3\sum_{k_1,k_2\in \mathbb{Z}}\frac{w^{-2(k_1+2k_2-1)}q^{k_1^2+k_2^2+k_1k_2+2k_1+2k_2+1}}{(1-w^{4}q^{2k_1+k_2+1})(1-w^{4}q^{k_2-k_1})}, \non
\end{eqnarray} 
\begin{eqnarray}
H_{3,C+f}(z,\tau;J_{1,0})&=&H_{1}(z,\tau) H_{2}(z,\tau) \left( \sum_{k\in \mathbb{Z}}
\frac{w^{-6k+6}q^{3k^2-\frac{1}{3}}}{1-w^{6}q^{3k-1}}+ \sum_{k\in \mathbb{Z}} \frac{w^{-6k}q^{3k^2+3k+\frac{2}{3}}}{1-w^{6}q^{3k+1}}\right) \non \\
&&+ \, H_{1}(z,\tau)^3\sum_{k_1,k_2\in \mathbb{Z}}\frac{w^{-2(k_1+2k_2-3)}q^{k_1^2+k_2^2+k_1k_2-\frac{1}{3}}}{(1-w^{4}q^{2k_1+k_2-1})(1-w^{4}q^{k_2-k_1})}, \non
\end{eqnarray} 
\begin{eqnarray}
H_{3,0}(z,\tau;J_{1,0})&=&2H_{1}(z,\tau) H_{2}(z,\tau)\left( \sum_{k\in \mathbb{Z}}  \frac{w^{-6k}q^{3k^2}}{1-w^{6}q^{3k}} \right)\non \\
&&+H_{1}(z,\tau)^3\sum_{k_1,k_2\in \mathbb{Z}}\frac{w^{-2(k_1+2k_2)}q^{k_1^2+k_2^2+k_1k_2}}{(1-w^{4}q^{2k_1+k_2})(1-w^{4}q^{k_2-k_1})}+H_{3}(z,\tau), \non
\end{eqnarray}
\begin{eqnarray}
H_{3,C}(z,\tau;J_{1,0})&=&H_{1}(z,\tau) H_{2}(z,\tau)\left( \sum_{k\in \mathbb{Z}}  \frac{w^{-6k+2}q^{3k^2-2k+\frac{1}{3}}}{1-w^{6}q^{3k-1}}+\sum_{k\in \mathbb{Z}}  \frac{w^{-6k-2}q^{3k^2+2k+\frac{1}{3}}}{1-w^{6}q^{3k+1}}  \right)\non \\
&&+H_{1}(z,\tau)^3\sum_{k_1,k_2\in \mathbb{Z}}\frac{w^{-2(k_1+2k_2+1)}q^{k_1^2+k_2^2+k_1k_2+k_1+k_2+\frac{1}{3}}}{(1-w^{4}q^{2k_1+k_2+1})(1-w^{4}q^{k_2-k_1})}. \non 
\end{eqnarray}

\providecommand{\href}[2]{#2}\begingroup\raggedright

\end{document}